\documentclass[a4paper,12pt]{article}

\usepackage{amsmath}
\usepackage{amsfonts}
\usepackage{amssymb}
\usepackage{amsthm}
\usepackage{amsopn}
\usepackage{amscd}

\usepackage[all,cmtip]{xy}

\usepackage{pstricks}

\usepackage[utf8]{inputenc}
\usepackage{courier}
\usepackage{url}
\usepackage{color}
\usepackage{subfigure}
\usepackage{xypic}
\usepackage{graphicx}

\usepackage[vcentering,dvips]{geometry}
\def\address#1#2{\begingroup
\noindent\parbox[t]{7.8cm}{%
\small{\scshape\ignorespaces#1}\par\vskip1ex
\noindent\small{\itshape E-mail address}%
\/: #2\par\vskip4ex}\hfill%
\endgroup}%

\author{Nathan Owen Ilten \& Hendrik Süß}
\title{Polarized Complexity-One $T$-Varieties}
\date{}
\newcommand{\st}{\;\mid\;}

\DeclareMathOperator{\tail}{tail}

\DeclareMathOperator{\spec}{Spec}

\DeclareMathOperator{\conv}{Conv}

\DeclareMathOperator{\PP}{\mathbb{P}}

\DeclareMathOperator{\WDiv}{Div}
\DeclareMathOperator{\Div}{div}

\DeclareMathOperator{\loc}{Loc}
\DeclareMathOperator{\Face}{face}
\DeclareMathOperator{\relint}{relint}
\DeclareMathOperator{\supp}{supp}

\DeclareMathOperator{\TCDiv}{T-CaDiv}
\DeclareMathOperator{\CaSF}{CaSF}
\DeclareMathOperator{\SF}{SF}

\DeclareMathOperator{\DIV}{Div}

\DeclareMathOperator{\vol}{vol}
\DeclareMathOperator{\hb}{HB}

\newcommand{\oneface}{\blacktriangleleft}
\newcommand{\CO}{\mathcal{O}}

\newcommand{\dfan}{\mathcal{S}}
\newcommand{\ZZ}{\mathbb{Z}}
\newcommand{\QQ}{\mathbb{Q}}
\newcommand{\NN}{\mathbb{N}}

\newcommand{\A}{\mathcal{A}}
\newcommand{\C}{\mathbb{C}}

\newcommand{\G}{\mathcal{G}}

\newcommand{\D}{\mathcal{D}}

\newcommand{\HP}{\mathcal{H}}

\newtheorem{lemma}{Lemma}[section]
\newtheorem{prop}[lemma]{Proposition}
\newtheorem{cor}[lemma]{Corollary}
\newtheorem{thm}[lemma]{Theorem}

\newtheorem*{thm*}{Theorem}

\theoremstyle{definition}
\newtheorem*{ex}{Example}
\newtheorem*{remark}{Remark}
\newtheorem*{defn}{Definition}

\newcommand{\fansydivnull}{%
 \psset{unit=1cm}
 \begin{pspicture}(-1.5,-.5)(3.5,.5)%
 \psset{linewidth=1pt}%
 \psline{<->}(-1,0)(3,0)
 \psset{dotstyle=|,linewidth=2pt}
 \psdots(0,0)(2,0)
\rput(0,0.3){\scriptsize$0$}	
\rput(1,0.3){\scriptsize$1$}	
\rput(2,0.3){\scriptsize$2$}	
\psset{linecolor=lightgray}
\psdot(1,0)
\end{pspicture}}

\newcommand{\fansydivother}{%
 \psset{unit=1cm}
 \begin{pspicture}(-1.5,-.5)(3.5,.5)%
 \psset{linewidth=1pt}%
 \psline{<->}(-1,0)(3,0)
 \psset{dotstyle=|,linewidth=2pt}
 \psdots(-.5,0)
 \rput(-.6,0.4){\scriptsize$-\frac{1}{2}$}	
\end{pspicture}}

\newcommand{\hnull}{%
 \psset{unit=1cm}
 \begin{pspicture}(-1.5,-4)(3.5,1)%
	 \psgrid[gridwidth=0.3pt,griddots=5,subgriddiv=1,gridlabels=5pt](-1,-3)(3,1)
 \psset{linewidth=1pt}%
 \psline{<->}(-1,0)(3,0)
 \psset{dotstyle=|,linewidth=2pt}
 \psdots(0,0)(2,0)
\psset{linecolor=lightgray}
\psdot(1,0)
\psset{linecolor=gray,linewidth=1pt}
\psline(-.5,-3)(0,-2)(1,-1)(2,-2)(2.5,-3)
\end{pspicture}}

\newcommand{\hother}{%
 \psset{unit=1cm}
 \begin{pspicture}(-1.5,-4)(3.5,1)%
	 \psgrid[gridwidth=0.3pt,griddots=5,subgriddiv=1,gridlabels=5pt](-1,-3)(3,1)
 \psset{linewidth=1pt}%
 \psline{<->}(-1,0)(3,0)
 \psset{dotstyle=|,linewidth=2pt}
 \psdots(-.5,0)
\psset{linecolor=gray,linewidth=1pt}
\psline(-1,-2)(-.5,0)(1,-3)
\end{pspicture}}

\newcommand{\hstarnull}{%
 \psset{unit=1cm}
 \begin{pspicture}(-2.5,-2.5)(2.5,2.5)%
	 \psgrid[gridwidth=0.3pt,griddots=5,subgriddiv=1,gridlabels=5pt](-2,-2)(2,2)
 \psset{linewidth=2pt}%
 \psline{[-]}(-2,0)(2,0)
\psset{linecolor=gray,linewidth=1pt}
\psline(-2,-2)(-1,0)(1,2)(2,2)
\end{pspicture}}

\newcommand{\hstarother}{%
 \psset{unit=1cm}
 \begin{pspicture}(-2.5,-2.5)(2.5,2.5)%
	 \psgrid[gridwidth=0.3pt,griddots=5,subgriddiv=1,gridlabels=5pt](-2,-2)(2,2)
 \psset{linewidth=2pt}%
 \psline{[-]}(-2,0)(2,0)
\psset{linecolor=gray,linewidth=1pt}
\psline(-2,1)(2,-1)
\end{pspicture}}

\begin{document}
\maketitle
\begin{abstract}
We describe polarized complexity-one $T$-varieties combinatorially in terms of so-called divisorial polytopes, and show how geometric properties of such a variety can be read off the corresponding divisorial polytope. We compare our description with other possible descriptions of polarized complexity-one $T$-varieties. We also describe how to explicitly find generators of affine complexity-one $T$-varieties.
\end{abstract}

\\


\section*{Introduction}
It is well known that there is a correspondence between polarized toric varieties and  lattice polytopes. The main result of this paper is to generalize this to the setting of normal varieties with effective complexity-one torus action, i.e. complexity-one $T$-varieties. In order to do so, we introduce so-called \emph{divisorial polytopes}. In short,
 a divisorial polytope on a smooth projective curve $Y$ in a lattice $M$ is a piecewise affine concave function $$\Psi=\sum_{P\in Y} \Psi_P\cdot P:\Box\to \DIV_\QQ Y$$ from some 
polytope in $M_\QQ$ to the group of $\QQ$-divisors on $Y$, such that 
\begin{enumerate}
	\item $\deg \Psi(u) > 0$ for $u$ from the interior of $\Box$,\label{item:con1}
	\item $\deg \Psi(u) > 0$ or  $\Psi(u) \sim 0$ for $u$ a vertex of $\Box$,\label{item:con2}
\item The graph of $\Psi_P$ has integral vertices for every $P \in Y$.
\end{enumerate}
We then show that, similar to the toric case, there is a correspondence between polarized complexity-one $T$-varieties and divisorial polytopes. We also describe how smoothness, degree, and Hilbert polynomial of a polarized $T$-variety can be determined from the corresponding divisorial polytope.

There are two other logical approaches to describing a polarized complexity-one $T$-variety.
Indeed, $T$-invariant Cartier divisors on complexity-one $T$-varieties have been described in \cite{lars08} in terms of \emph{divisorial fans} and support functions, along with a characterization of ampleness. On the other hand, a sufficiently high multiple of some polarizing line bundle gives a map to projective space such that the corresponding affine cone is a complexity-one $T$-variety describable by a \emph{polyhedral divisor} $\D$. We compare these two approaches with our divisorial polytopes, showing how to pass from one description to another.

We also present two other results. Firstly, we show how the complicated combinatorial data of a divisorial fan used to describe a general $T$-variety can be simplified to a so-called \emph{marked fansy divisor} for complete complexity-one $T$-varieties. Secondly, we address the problem of finding minimal generators for the multigraded $\C$-algebra corresponding to a polyhedral divisor $\D$ on a curve. This then gives us a method to determine whether projective embeddings of complexity-one $T$-varieties are in fact projectively normal.

We begin in section \ref{sec:tvar} by recalling the construction of $T$-varieties from \cite{MR2426131}. We specialize to the complexity-one case, and introduce marked fansy divisors. In section \ref{sec:div}, we then recall the description of $T$-invariant Cartier divisors. Section \ref{sec:dpoly} is dedicated to divisorial polytopes. Here we prove the correspondence between divisorial polytopes and polarized complexity-one $T$-varieties, and discuss properties of divisorial polytopes. In section \ref{sec:cones}, we then compare support functions and divisorial polytopes with polyhedral divisors corresponding to affine cones. Finally, in section \ref{sec:gen} we describe how to find minimal generators for affine complexity-one $T$-varieties.

We should remark that while this paper only looks at complexity-one $T$-varieties, we believe the correspondence between polarized $T$-varieties and divisorial polytopes should generalize to higher-complexity torus actions. To generalize the above definition of divisorial polytopes, we replace $Y$ by any normal projective variety, and the degree conditions in \ref{item:con1} and \ref{item:con2} are replaced respectively by ampleness and semiampleness. 

\section{Polyhedral Divisors and $T$-Varieties}\label{sec:tvar}
We recall several notions from~\cite{MR2426131}, and will then specialize these to the case of complexity-one $T$-varieties. As usual, let $N$ be a lattice with dual $M$ and let $N_\QQ$ and $M_\QQ$ be the associated $\QQ$ vector spaces.
For any polyhedron $\Delta\subset N_\QQ$, let $\tail(\Delta)$ denote its tailcone, that is, the cone of unbounded directions in $\Delta$. Thus, $\Delta$ can be written as the Minkowski sum of some bounded polyhedron and its tailcone. For any polyhedron $\Delta\subset N_\QQ$ and vector $u$ in the dual of its tailcone, let $\Face(\Delta,u)$ be the set of $\Delta$ on which $u$ attains its minimum. A \emph{face} of $\Delta$ is then defined to be any subset of $\Delta$ of the form $\Face(\Delta,u)$, or the empty set.

Let $Y$ be a normal semiprojective variety over $\C$ and let $\sigma\subset N_\QQ$ be a pointed polyhedral cone. By $\sigma^\vee$ we denote the dual cone of $\sigma$.
 \begin{defn}A \emph{polyhedral divisor} on $Y$ with tail cone $\sigma$ is a 
formal finite sum
$$\D = \sum_P \Delta_P \cdot P,$$
where $P$ runs over all prime divisors on $Y$ and $\Delta_P$ is a polyhedron with tailcone $\sigma$.
Here, finite means that only finitely many coefficient differ from the 
tail cone. Note that the empty set is also allowed as a coefficient.
If $Y$ is a complete curve, we define the \emph{degree} of a polyhedral divisor by $$\deg \D := \sum_P \Delta_P $$ where summation is via Minkowski addition. If $Y$ is an affine curve, we define the degree as $\deg \D=\emptyset$.
\end{defn}
We can evaluate a polyhedral divisor for every element $u \in \sigma^\vee \cap M$ via
$$\D(u):=\sum_P \min_{v \in \Delta_P} \langle v , u \rangle P$$
 in order to obtain an ordinary divisor $\D(u)$ on the locus of $\D$, which is defined as $\loc \D := Y \setminus \left( \bigcup_{\Delta_P = \emptyset} P \right)$.

\begin{defn}
A polyhedral divisor $\D$ is called {\em proper} if $\D(u)$ is a semiample $\QQ$-Cartier divisor for all $u\in\sigma^\vee$, and if $\D(u)$ is big for all $u$ in the interior of $\sigma^\vee$. If $Y$ is a curve, note that $\D$ is proper exactly when $\deg \D \subsetneq \sigma$, and for all $u\in\sigma^\vee$ with $\min_{v\in \deg \D} \langle v,u\rangle= 0$ it follows that 
a multiple of $\D(u)$ is principal.
\end{defn}

To a proper polyhedral divisor we associate an $M$-graded $\mathbb{C}$-algebra and consequently
an affine scheme admitting a $T^N=N\otimes \C^*$-action:
$$X(\D):= \spec \bigoplus_{u \in \sigma^\vee \cap M} H^0(Y,\D(u)).$$
This construction gives a normal variety of dimension $\dim N_\QQ+\dim Y$ together with an effective $T^N$-action. 

\begin{remark}
	If $\D$ is a non-proper polyhedral divisor, we can still associate an $M$-graded $\mathbb{C}$-algebra as above, and consequently an affine scheme $X(\D)$ with $T^N$-action. However, the resulting algebra need not be finitely generated; similarly, we can't say anything about the dimension of $X(\D)$ or the effectiveness of the $T^N$-action.
\end{remark}

In order to glue together the affine varieties with $T^N$-action, we require some further definitions:
\begin{defn}
Let $\D=\sum_P \Delta_P \cdot P$, $\D'=\sum_P \Delta'_P \cdot P$ be two polyhedral divisors on $Y$ with tail cones $\sigma$ and $\sigma'$.
\begin{itemize}
\item We define their \emph{intersection} by $$\D \cap \D' := \sum_P (\Delta_P \cap \Delta'_P) \cdot P.$$
\item We say  $\D' \subset \D$ if $\Delta'_P\subset\Delta_P$ for every point $P \in Y$.
\item For $y\in Y$ a not necessarily closed point, we call $\D_y:=\sum_{P\ni y} \Delta_P$ the slice of $\D$ at $P$ and denote it by $\D_y$ as well.
\end{itemize}
\end{defn}

If $\D' \subset \D$ and both are proper then we have an inclusion
$$\bigoplus_{u \in \sigma^\vee \cap M} H^0( Y,\D'(u))
\supset \bigoplus_{u \in \sigma^\vee \cap M} H^0( Y,\D(u))$$ which corresponds to a dominant morphism $X(\D') \rightarrow X(\D)$. We say that $\D'$ is a \emph{face} of $\D$, written $\D'\prec\D$, if this morphism is an open embedding. 

\begin{defn}
	A {\em divisorial fan} is a finite set $\dfan$ of proper polyhedral divisors such that for $\D,\D' \in \Xi$ we have $\D \succ \D' \cap \D \prec \D'$ with $\D' \cap \D$ also in $\dfan$. The {\em tailfan} of $\dfan$ is the set of all $\tail(\D)$ for $\D\in\dfan$.
For  a not necessarily closed point $y\in Y$, the polyhedral complex $\dfan_y$ defined by the polyhedra $\D_y$, $\D \in \Xi$ is called a {\em slice} of $\dfan$.
$\dfan$ is called {\em complete} if all slices $\dfan_y$ are complete subdivisions of $N_\QQ$ and $Y$ is complete.
\end{defn}

We may glue the affine varieties $X(\D)$ via
$$X(\D) \leftarrow X(\D \cap \D') \rightarrow X(\D').$$
This construction yields a normal scheme $X(\dfan)$ of dimension $\dim N_\QQ+\dim Y$ with an effective torus action by $T^N$; furthermore, $X(\dfan)$ is complete if and only if $\dfan$ is complete. Note that all normal varieties with effective torus action can be constructed in this manner. 

For the remainder of the section we will restrict to the case that $Y$ is a curve; this is thus the case of complexity one $T$-varieties. As we had already seen, the criterion for properness of a polyhedral divisor simplifies nicely. This is true as well for the face relation.  Let $\D,\D'$ be a polyhedral divisors on a curve $Y$ with $\D$ proper. In this case, we say $\D'\oneface \D$ if $\Delta'_P$ is a face of $\Delta_P$ for every point $P \in Y$ and $\deg \D \cap \sigma' = \deg \D'$. We then have the following proposition:

\begin{prop}\label{prop:faces}
Let $\D,\D'$ be a polyhedral divisors on a curve $Y$ with $\D$ proper. Then $\D'\prec \D$ if and only if $\D'\oneface\D$.
\end{prop}

We shall need several lemmas to prove the proposition.

\begin{lemma}[Refinement lemma]
\label{sec:lem-refine}
  Let $\D$ be a polyhedral divisors with affine locus $Y$ and let $\{U_i\}_{i \in I}$ be an affine covering of $Y$. The polyhedral divisors $\D + \emptyset \cdot (Y \setminus U_i) =: \D|_{U_i} \prec \D$ define open subsets $X(\D|_{U_i}) \hookrightarrow X(\D)$, which cover $X:=X(\D)$. 
\end{lemma}
\begin{proof}
  Every global section $f \in \Gamma(\CO_Y)$ gives rise to a section $f \in \Gamma(X, \CO_{X})_0 = \Gamma(\CO_Y)$. By \cite[Proposition 3.1.]{MR2426131} we have $X_f = X(\D+\emptyset \cdot \Div(f))$. Hence, for principal open subsets $U_i = Y_{f_i}$ the claim follow immediately.

Since $Y$ is affine, by refining  we can pass to a covering $\{U'_j\}_{j \in J}$ of principal open subsets and corresponding polyhedral divisors $\D|_{U'_j}$. Now the $X(\D|_{U'_j})$ define open subsets of $X(\D|_{U_i})$ and of $X$ as well and cover them by the above conclusion. Since the inclusions $X(\D|_{U'_j}) \hookrightarrow X$ factor through the $X(\D|_{U_i})$s these already define an open covering of $X$.
\end{proof}

\begin{lemma}[\cite{MR2426131}, lemma 6.8.]
\label{sec:lem-open-semiample}
  Assume that $\loc \D' = \loc \D \setminus Z$ and $\D'_P =\Face(\D_P,u)$ for some $u \in  \sigma^\vee$ and all $P \in \loc \D'$. Than $\D' \subset \D$ defines an open embedding if there is a semiample divisor $E$ with support $Z$ and $k\cdot \D(u) - E$ semiample for $k \gg 0$.
\end{lemma}

\begin{lemma}\label{lemma:ourfaceistheirface}
  If $\D$ is proper and $\D'\oneface \D$, then $\D'$ is proper, too, and the corresponding morphism $i:X(\D') \rightarrow X(\D)$ is an open embedding.
\end{lemma}

\begin{proof}
First we check the properness of $\D'$. For affine locus there is nothing to prove. So we shall assume that $\D'$ and $\D$ have complete loci.   If we have $\deg \D' = \deg \D \cap \sigma'$, by the properness of $\D$ we get $\deg \D' \subsetneq \sigma'$. 
Now for every $u' \in (\sigma')^\vee$ there exists a decomposition $u'=u-u''$ such that $u \in \sigma^\vee$ and $u'' \in \sigma^\vee \cap (\sigma')^\perp$. First, note that $u \mapsto \D'(u)$ is by definition a concave map. Hence, $\D'(u-u'') \geq \D'(u) + \D'(-u'')$ holds. 
 The inclusion  $\deg \D' \subset \sigma'$ implies the equality $\D'(-u'') = -\D'(u'')$. Moreover, because of the inclusion $\D' \subset \D$ we conclude $\D'(u) \geq \D(u)$ and $\D'(u'') \geq \D(u'')$. But since $\deg \D'(u'') = 0$ it follows that $\D(u'')=\D'(u'')$. In particular, $\deg \D(u'')=0$. All together we get $\D(u') \geq \D(u)-\D(u'')$. Since $\D(u),-\D(u'')$ are semi-ample by the properness of $\D$, the same is true for $\D'(u')$.

We now check that $i$ is an open embedding. Suppose first that $\D$  has affine locus. We first assume additionally that $\Delta'_P=\Face(\Delta_P,u)$ for some $u \in \sigma^\vee$ and all $P \in \loc \D'$. Then $i$ is indeed an open embedding by the previous lemma, since on a affine variety every divisor is semi-ample.
Dropping the additional assumption, we may chose an open covering $\{U_j\}_{j \in J}$ of $Y$ and refine $\D'$ by $\D'|_{U_j}$ as in lemma~\ref{sec:lem-refine} such that $(\D'|_{U_j})_P = \Face(\D_P,u)$ for some $u$ and all $P \in Y$. Now we infer that $X(\D'|_{U_j}) \rightarrow X(\D)$ is an open embedding for every $j$ and by the refinement lemma we are done.

For $\D$ of complete locus and $\D'$ not we first also assume that $\Delta'_P=\Face(\Delta_P,u)$ for some $u \in \sigma^\vee$ and all $P \in \loc \D'$; in this case we again obtain our result by applying lemma~\ref{sec:lem-open-semiample}. We may choose any effective divisor with support $Y \setminus \loc \D'$.
The relation $\D' \prec \D$ implies that $\deg \D(u) > 0$. Hence, $\deg (k\cdot \D(u) - E) > 0$ for $k \gg 0$. For the general case we may once again refine $\D'$ as above to conclude that $i$ is an open embedding.

Finally, if $\D$ and $\D'$ both have complete loci, $\deg \D' = \deg D \cap \sigma'$ implies that for any $u$ with $\sigma'=\Face(\sigma,u)$, we have  $\Delta'_P=\Face(\Delta_P,u)$ for all $P \in Y$. Now we can again use lemma~\ref{sec:lem-open-semiample} with $Z=\emptyset$ and $E=0$ since $\D(u)$ is semiample by the properness condition.
\end{proof}

\begin{lemma}\label{lemma:theirfaceisourface}
  Let $\D,\D'$ be two proper polyhedral divisors with $\D' \prec \D$. Then $\deg \D' = \sigma' \cap \deg \D$.
\end{lemma}

\begin{proof}
If $\loc (\D)$ is affine, the claim is immediate. We can thus assume that $\loc \D$ is complete for the rest of the proof. Recall from proposition 3.4 and definition  5.1 of 
\cite{MR2426131} that $\D' \prec \D$ is equivalent to the following condition:
\begin{center}\begin{tabular}{p{14cm}}
	For every $y\in Y$  there exists $w_y\in\sigma^\vee\cap M$ and a $D_y$ in the linear system $|\D(w_y)|$, such that $y\notin\supp(D_y)$, $\D_y'=\Face(\D_y, w_y)$, and $\Face(\D_v',w_y)=\Face(\D_v, w_y)$ for all $v\in Y\setminus \supp(D_y)$. 
\end{tabular}
\end{center}
Now suppose first that $\loc (\D')$ is affine. Then we must show that 
\begin{equation}
	\sigma'\cap \deg \D=\emptyset.\label{eqn:affine-case}
\end{equation}
For each $w_y$ and $D_y$ as above, the support of $D_y$ cannot be empty, since otherwise $\loc(\D')=\loc(\D)$. In particular, $\deg(\D(w_y))>0$. Now, choosing $y$ to be some general point gives us $w_y$ such that $\sigma'=\Face(\sigma,w_y)$ with $(\deg \D)(w_y)>0$. But this is equivalent to \eqref{eqn:affine-case}, since $\langle \sigma', w_y\rangle=0$.

Suppose instead that $\loc (\D')$ is complete. Given an element $v \in \deg \D'$ it follows from the properness of $\D'$ that 
$v \in \sigma'$ and from $\D' \subset \D$ that $v \in \deg \D$. Hence, we get $\deg \D' \subset  \deg \D \cap \sigma'$. For the other inclusion we choose an element $v= \sum_y v_y \in \deg \D \cap \sigma'$, with $v_y \in \D_y$. Now we choose an element $u \in \sigma^\vee$ such that $\D'_z = \Face(\D_z,u)$ for some $z \in Y$. This implies that $\sigma' = \Face(\sigma,u)$.  Since $\langle v_y, u \rangle \geq \min \langle \D_y, u \rangle$ holds we get $0 = \sum_y \langle v_y, u \rangle \geq \sum_y \min \langle \D_y, u \rangle \geq 0$, where first equality follows from the fact that $v \in \sigma'$ and the last inequality  from the properness of $\D$. Hence, $\langle v_y, u \rangle = \min \langle \D_y, u \rangle$ holds for every $y \in Y$ and for $y=z$ we get $v_z \in \D'_z = \Face(\D_z,u)$. Since this is true for every $z \in Y$ we conclude that $v = \sum_z v_z \in \sum_z \D'_Z = \deg \D'$.
\end{proof}

\begin{proof}[Proof of proposition \ref{prop:faces}]
	The proposition follows from the above lemmas. Indeed, lemma \ref{lemma:ourfaceistheirface} covers one direction. The other direction follows from lemma \ref{lemma:theirfaceisourface} coupled with the fact that if $\D'\prec \D$, definition 5.1 of \cite{MR2426131} ensures that $\Delta'_P$ is a face of $\Delta_P$ for every point $P \in Y$.
\end{proof}

Different divisorial fans $\dfan,\dfan'$ can in fact yield the same $T$-variety $X(\dfan)=X(\dfan')$. The differing divisorial fans simply correspond to different open affine coverings. On the other hand,  divisorial fans with identical slices might yield differing $T$-varieties, even in the complexity-one case. However, for complete complexity-one $T$-varieties, we can save the situation via the following definition:

\begin{defn}
	A \emph{marked fansy divisor} on a curve $Y$ is a formal sum $\Xi=\sum \Xi_P\cdot P$ together with a fan $\Sigma$ and some subset $C\subset \Sigma$, such that 
	\begin{enumerate}
		\item $\Xi_P$ is a complete polyhedral subdivision of $N_\QQ$ and $\tail(\Xi_P)=\Sigma$ for all $P\in Y$;\label{item:complete}
		\item For full-dimensional $\sigma\in C$ the polyhedral divisor $\D^\sigma=\sum \Delta_P^\sigma \cdot P$ is proper, where $\Delta_P^\sigma$ is the unique element of $\Xi_P$ with $\tail(\Delta_P^\sigma)=\sigma$.\label{item:proper}
                \item For $\sigma \in C$ of full dimension and $\tau \prec \sigma$, we have $\tau \in C$ if and only if $\deg \D^\sigma \cap \tau \neq \emptyset$.\label{item:open1}
                \item If $\tau \prec \sigma$ and $\tau \in C$, then $\sigma \in C$.\label{item:open2}
	\end{enumerate}

	We say that the elements of $C$ are \emph{marked}. The support of a fansy divisor is the set of points $P \in Y$, where $\Xi_P$ differs from the tailfan $\Sigma$.
\end{defn}

Now, given any complete divisorial fan $\dfan$ on $Y$, we can associate a marked fansy divisor, by setting $\Xi=\sum \dfan_P\cdot P$ and adding marks to the tailcones of all $\D\in\dfan$ with complete locus. We call this marked fansy divisor $\Xi(\dfan)$.
\begin{prop}\label{prop:markedfansydiv}
For any marked fansy divisor $\Xi$, there exists a complete divisorial fan $\dfan$ with $\Xi=\Xi(\dfan)$. If for two divisorial fans $\dfan,\dfan'$ we have that $\Xi(\dfan)=\Xi(\dfan')$, then it follows that $X(\dfan)=X(\dfan')$.
\end{prop}
\begin{proof}
  Assume that $\Xi$ is supported at $P_1,\ldots,P_r$. We construct a divisorial fan as follows. Consider the set 
\[       S = \{\D^\sigma \st \sigma \in C\} \cup  \{\Delta \cdot P_i + \textstyle  \sum_{j\neq i} \emptyset \cdot P_j \st \Delta \in \Xi_{P_i}^{(n)},\; \tail(\Delta) \notin C\} 
\]

Now we get the divisorial fan $\mathcal{S}$ generated by $S$ by adding all intersection of the polyhedral divisors in $S$.
This is indeed a divisorial fan since \ref{item:proper} ensures that the polyhedral divisors with maximal tailcone are proper and \ref{item:open1} \& \ref{item:open2} ensures that the intersection of two polyhedral divisors is face of both of them. Obviously we have $\Xi(\dfan)=\Xi$. 

Now let $\dfan'$ be another divisorial fan with $\Xi(\dfan')=\Xi$.  We get a common refinement $\dfan''$ of $\dfan$ and $\dfan'$ by considering all mutual intersections of divisors in  $\dfan$ and $\dfan'$. To get the corrects marks the polyhedral divisors with complete locus in $\dfan'$  must be exactly the $\{\D^\sigma \mid \sigma \in C\}$. Hence, only polyhedral divisors with affine locus get refined. Now the claim follows by the refinement lemma.
\end{proof}

By the above proposition, we can thus define $X(\Xi)$ to be $X(\dfan)$ for any $\dfan$ with $\Xi=\Xi(\dfan)$. Furthermore, every complete complexity-one $T$-variety can be described via a marked fansy divisor. We thus can avoid divisorial fans and work instead with the somewhat more handy notion of marked fansy divisors.

\begin{figure}[htbp]
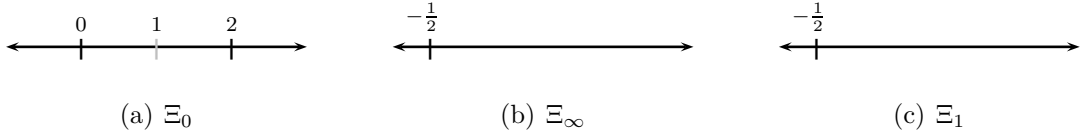

    \centering
    \subfigure[$\Xi_0$]{\fansydivnull}
    \subfigure[$\Xi_\infty$]{\fansydivother}
    \subfigure[$\Xi_1$]{\fansydivother}
    \caption{The fansy divisor for a log Del-Pezzo surface}\label{fig:fansydiv}
\end{figure}

\begin{ex}
	The subdivisions demarcated by the black lines in figure \ref{fig:fansydiv} together with marks for both $\QQ_{\geq 0}$ and $\QQ_{\leq 0}$ give a marked fansy divisor $\Xi$ on $\mathbb{P}^1$ with $X(\Xi)$ equal to the unique log Del-Pezzo surface of degree 2 with two $A_3$ and one $A_1$ singularities, see \cite{suess08}. By further subdividing at the gray line, we get a marked fansy divisor $\Xi'$. The corresponding $T$-variety comes together with a natural map $\varphi:X(\Xi')\to X(\Xi)$ which is a resolution of the $A_1$ singularity.
\end{ex}

\section{Invariant Cartier Divisors}\label{sec:div}
Invariant Cartier divisors on complexity-one $T$-varieties have been described in \cite{lars08} in combinatorial terms. We recall this description here, specializing to complete $T$-varieties and replacing divisorial fans with marked fansy divisors. For any piecewise affine continuous function $f:N_\QQ\to \QQ$ set $f^0(v)=\lim_{k\to \infty}  f(k\cdot v)/k$ for any $v\in N_\QQ$. We call $f^0$ the linear part of $f$. 
Consider now some complete marked fansy divisor $\Xi$ on a smooth projective curve $Y$ with tailfan $\Sigma$. 

\begin{defn}
By $\SF(\Xi)$ we denote  the set of all formal sums of the form
$$h=\sum_{P\in Y} h_P\otimes P$$ 
where $h_P:N_\QQ\to \QQ$ are continuous functions such that:
\begin{enumerate}
	\item $h_P$ is piecewise affine with respect to the subdivision $\Xi_P$;
	\item $h_P$ is integral, that is, if $k\cdot v$ is a lattice point for $k\in \NN$, $v\in N$, then $k\cdot h_P(v)\in\ZZ$;
	\item $h_P^0$ doesn't depend on $P$; we call this the linear part of $h$ and denote it by $h^0$. 
	\item $h_P\neq h^0$ for only finitely many $P$.
\end{enumerate}
We call an element of $\SF(\Xi)$ a support function.
\end{defn}

Consider $\sigma$ a full-dimensional cone in $\Sigma$. We define 
$$h_{|\sigma}(0)=\sum_P a_P \cdot P$$
where the $a_P$ are determined by writing ${h_P}_{|\Delta_P^\sigma}(v)=\langle v,u\rangle +a_P$. We then define $\CaSF(\Xi)$ to consist of all $h\in \SF(\Xi)$ such that for every marked $\sigma\in\Sigma$,  $h_{|\sigma}(0)$ is a principal divisor on $Y$.
Both $\SF(\Xi)$ and $\CaSF(\Xi)$ have a natural group structure. There is a group isomorphism from $\CaSF(\Xi)$ to the group $\TCDiv(X(\Xi))$ of $T$-invariant Cartier divisors on $X(\Xi)$; we denote the divisor associated to $h$ by $D_h$. We call a support function $h$ \emph{ample} if $D_h$ is ample.

\begin{prop}[\cite{lars08} 3.28]\label{prop:ample}
Consider $h\in \CaSF(\Xi)$. Then $h$ is ample if and only if $h$ is strictly concave, and if for all unmarked $\sigma\in\Sigma$ with $\sigma$ full-dimensional, $- \deg h_{|\sigma}(0) > 0$.
\end{prop}

For a support function $h\in\CaSF(\Xi)$, we define its weight polytope $\Box_h\subset M_\QQ$ by
$$
\Box_h=\{u\in M_\QQ\ |\ h^0(v)\geq \langle v,u\rangle\ \forall\ v\in N_\QQ\}.
$$
We then define the dual of $h$ to be the piecewise affine concave function $h^*:\Box_h\to\DIV_\QQ Y$ given by
\begin{equation*}
	h^*=\sum_{P\in Y} h_P^*\cdot P\qquad\qquad h_P^*(u)=\min_{\substack{v\in \Xi_P\\ v\ \textrm{vertex}}} \langle v,u\rangle -h_P(v).
\end{equation*}

\begin{prop}[\cite{lars08} 3.23]\label{prop:gs}
For $h\in \CaSF(\Xi)$ and $X=X(\Xi)$, we have
\begin{equation*}
	H^0(X,D_h)_u=\begin{cases} 
		H^0(Y,h^*(u))& u\in\Box_h\cap M\\
	0 & u\notin\Box_h\cap M.\\\end{cases}
\end{equation*}
\end{prop}

\begin{figure}[htbp]
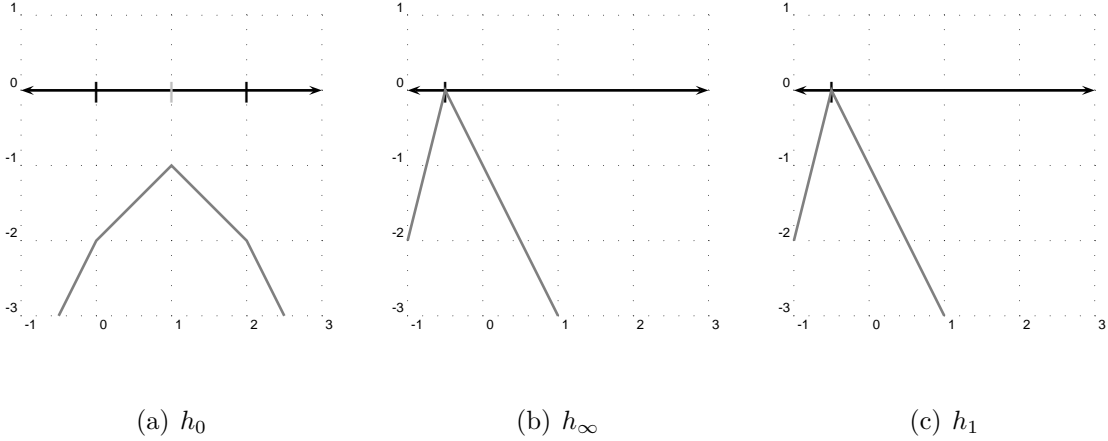

    \centering
    \subfigure[$h_0$]{\hnull}
    \subfigure[$h_\infty$]{\hother}
    \subfigure[$h_1$]{\hother}
    \caption{The support function for $\varphi^*(-2K)-E$}\label{fig:supportfunction}
\end{figure}

\begin{ex}
	Continuing the example from the previous section, the support function $h$ pictured in figure \ref{fig:supportfunction} corresponds to a divisor on $X(\Xi')$. In fact, using the formula for canonical divisors in \cite{lars08}, one easily checks that $D_h=\phi^*(-2K)-E$, where $K$ is a canonical divisor on $X(\Xi)$ and $E$ is the exceptional divisor of $\varphi$. Using proposition \ref{prop:ample}, we easily check that $D_h$ is ample.
\end{ex}

\section{Divisorial Polytopes}\label{sec:dpoly}
\begin{defn}
A divisorial polytope $(\Psi,\Box)$ consists of a lattice polytope $\Box\subset M_\QQ$ and a piecewise affine concave function $$\Psi=\sum \Psi_P\cdot P:\Box\to \DIV_\QQ Y,$$ 
such that
\begin{enumerate}
\item $\deg \Psi(u) > 0$ for $u$ in the interior of $\Box$;
\item $\deg \Psi(u) > 0$ or  $\Psi(u) \sim 0$ for $u$ a vertex of $\Box$;
\item For all $P\in Y$, the graph of $\Psi_P$ is integral, i.e. has its vertices in $M \times \ZZ$.
\end{enumerate}
We often will call the pair $(\Psi,\Box)$ simply $\Psi$.
\end{defn}

The set of divisorial polytopes for fixed lattice $M$ and fixed curve $Y$ form in fact a natural semigroup. Indeed, for divisorial polytopes $(\Psi',\Box')$, $(\Psi'',\Box'')$, we define $\Psi'+\Psi'':(\Box'+\Box'')\to\Div_\QQ Y$ via
$$(\Psi'+\Psi'')(u)=\sum \max_{\substack{u'+u''=u\\ u'\in\Box',\ u''\in \Box''}}\Psi_P'(u')+\Psi_P'(u'').$$
The neutral element is then obviously the constant function $0$ on the $0$ polytope. For any $k\in \NN$ and divisorial polytope $\Psi$, we similarly define $k\cdot \Psi$ to be the $k$-fold sum of $\Psi$.

Before proceeding to associate a marked fansy divisor and support function to a divisorial polytope, we shortly recall the toric construction of a fan from a polytope. Consider a polytope $\Box \subset M_\QQ$. For every face $F$ of $\Box$ we consider the cone $\sigma_F \subset N_\QQ$ consisting of all $v$, such that $\langle v, \cdot \rangle$ obtains its minimum at $F$. These are exactly the inner normal vectors at $F$. The cones $\sigma_F$ form a fan---the normal fan of $\Box$. This fan can be seen as spanned by the regions where the piecewise linear function $\min_{u\in \Box} \langle u, \cdot \rangle$ is linear.
The corresponding face to a given cone $\sigma$ of the normal fan we denote by $F_\sigma$. The described correspondence between faces of $\Box$ and cones of the normal fan is inclusion reversing and map faces of dimension $r$ to cones of dimension $\dim N -r$. Moreover, we have $\langle u-u' \mid u,u' \in F \rangle = \sigma_F^\perp$.

\begin{prop}
Let $\Xi$ be a marked fansy divisor, and $g,h\in\CaSF(\Xi)$ ample. Then 
\begin{enumerate}
	\item $(g^*,\Box_g)$ and $(h^*,\Box_h)$ are divisorial polytopes;
	\item $(g+h)^*=g^*+h^*$;
	\item If $g^*=h^*$, then $g=h$.

\end{enumerate}
	\end{prop}

\begin{proof}
  Every  maximal cone $\sigma \in \tail \Xi$ corresponds to a vertex $u_\sigma$ of $\Box_g$. Moreover the concaveness of $g$ implies that $-g|_{\sigma}(0) = g^*(u_\sigma)$. Now the ampleness condition on $g$ implies that 
$\deg g^*(u_\sigma)>0$ for unmarked $\sigma$ and the Cartier condition implies that $g^*(u_\sigma) \sim 0$ for marked $\sigma$. Since $g_P$ is integral the same is true for the graph of $g^*_P$. The first claim follows.
The remaining two claims are easily seen from the definitions of $g^*$ and $h^*$.
\end{proof}

We now show how to associated a marked fansy divisor and support function to a divisorial polytope $(\Psi,\Box)$. We begin by setting $\Psi_P^*(v)=\min_{u\in \Box} (\langle v,u\rangle -\Psi_P(u))$, which is a piecewise affine concave function on $N_\QQ$. Now let $\Xi_P$ be the polyhedral subdivision of $N_\QQ$ induced by $\Psi_P^*$ and take $\Xi=\sum \Xi_P\cdot P$. Furthermore, we add a mark to an element $\sigma\in\tail(\Xi)$ if $(\deg \circ \Psi)|_{F_\sigma} \equiv 0$, where $F_\sigma \prec \Box$ is the face where $\langle \cdot, v \rangle$ takes its minimum for all $v \in \sigma$.

\begin{thm}
Using notation from the above construction, $\Xi$ is a marked fansy divisor, and $\Psi^*=\sum \Psi_P^*\cdot P\in \CaSF(\Xi)$ is a support function satisfying the properties that 
\begin{enumerate}
	\item $\Psi^*$ is ample;
	\item $(\Psi^{* *},\Box_{\Psi^*})=(\Psi,\Box)$.
\end{enumerate}
Thus, the above construction induces a correspondence between divisorial polytopes and pairs $(X,\mathcal{L})$ of complexity-one varieties with an invariant ample line bundle.
\end{thm}

\begin{proof}
  The maximal polyhedra in $\Xi_P$ consist of those $v$ such that  the minimum of $(\langle v ,\cdot \rangle -\Psi_P(\cdot))$ is realized by the same vertex $u \in \Box$. We will denote such a polytope by $\Delta_P^u$.
For $w \in \Delta^u_P$ and $v \in \sigma_u$ we obviously get $v+w \in \Delta_u$. Hence, the tail fan of $\Xi_P$ is exactly the normal fan of $\Box$.

Next we have to check that the properties \ref{item:proper}-\ref{item:open2} for the markings of a fansy divisor are fulfilled. For condition~\ref{item:open2}  we have to check that for a marked cone all cones that contain it are marked, too. By our setting of marks this corresponds to the fact that if $(\deg \circ \Psi)|_{F} \equiv 0$ holds this is also true for all faces of $F$.

We now turn to conditions \ref{item:proper} and \ref{item:open1}. Fix some vertex $u$ of $\Box$ with $\deg \Psi(u) = 0$, and let $\sigma$ be the corresponding cone. We now consider some $v \notin \sigma$. 
This implies that $\langle v, \cdot \rangle$ does not get minimal at $u$. Since $\deg \Psi(u') \geq 0$ the minimum of $(\langle v , \cdot \rangle -\deg \Psi(\cdot))$ also cannot be realized at $u$ and $v \notin \sum_P \Delta_P^u = \deg \D^\sigma$. Since $\deg \Psi(u')>0$ for some $u'$ we also infer that $0 \notin \deg \D^\sigma$. Hence, we obtain $\deg \D^\sigma \subsetneq \sigma$.

We now assume that  $\deg \D^\sigma \cap \tau \neq \emptyset$ for some face $\tau$ of $\sigma$. We choose some $v \in \deg \D^\sigma \cap \tau$. Since $v \in \deg \D^\sigma$ we know that $(\langle v , \cdot \rangle -\deg \Psi(\cdot))$ obtains its minimum at $u$. Hence, $(\langle v , u' \rangle -\deg \Psi(u')) \geq (\langle v , u \rangle -\deg \Psi(u))$ for any element $u' \in \Box$.  For $u' \in F_\tau$ we get $\langle v , u \rangle = \langle v , u' \rangle$, since $u'-u\in \tau^\perp$. This implies that $\deg \Psi(u')=\deg \Psi(u)=0$. Hence, $(\deg \circ \Psi)|_{F_\tau} \equiv 0$. By construction of $\Xi$ we thus have that $\tau$ is marked, too. For the other direction let us assume that $(\deg \circ \Psi)|_{F_\tau} \equiv 0$ for some $\tau \prec \sigma \in C$. We choose any interior point $v \in \relint \tau$.  We know that the elements of $\deg \D^\sigma$ are those $v$ such that $(\langle v ,\cdot \rangle -\Psi_P(\cdot))$ takes its minimum at $u = F_\sigma$. For any $u'' \notin F_\tau$ we then get $\langle v , u'' \rangle > \langle v , u \rangle$ and hence 
 $(\langle k \cdot v , u'' \rangle -\deg \Psi(u')) > (\langle k \cdot v , u \rangle -\deg \Psi(u))$ for $k \gg 0$. Since $\deg \Psi(u') = \deg \Psi(u)$ holds for $u' \in F_\tau$ we conclude that  $k \cdot v \in \deg \D^\sigma \cap \tau$. This proves  \ref{item:open1}.

 To finish the proof of \ref{item:proper}, assume that $\deg \D^\sigma(w) = 0$. We have to show that a multiple of $\D^\sigma(w)$ is principal. Without loss of generality we may assume that $\tau=\Face(\sigma,u')$ is a facet. Thus $\tau=\sigma \cap \sigma'$ for another maximal cone $\sigma'$ with corresponding vertex $u'$. Now $w=\lambda \cdot (u'-u)$ and $u'-u \in \tau^\perp$. By the last step we know that $\deg \D^{\sigma'} \cap \tau \neq \emptyset$ and hence 
 for every $P$ there is a $v_P \in \Delta_P^u \cap \Delta_P^{u'}$. This implies that $(\langle v_P ,u \rangle -\Psi_P(u))=(\langle v_P ,u' \rangle -\Psi_P(u'))$. Thus, we get $\min\langle \Delta^u_P, u'-u\rangle = \langle v_P ,u'-u \rangle= \Psi_P(u')-\Psi_P(u)$. Now condition \ref{item:proper} follows from the fact that $\Psi(u)$ and $\Psi(u')$ are principal, since $\D^\sigma(\lambda \cdot (u'-u))= \lambda \cdot (\Psi(u')-\Psi(u))$.

$\Psi^*_P$ is strictly concave on $\Xi_P$ by the construction of $\Xi$. Furthermore, for 
$\sigma$ maximal we have $\Psi^*|_\sigma(0)=-\Psi(u_\sigma)$. Hence, the ampleness follows from the condition $\deg \Psi(u)>0$ for $\sigma_u$ unmarked. Finally, a simple calculation shows that $(\Psi^{**},\Box_{\Psi^*})=(\Psi,\Box)$.
\end{proof}

\begin{remark}
  Two divisorial polytopes $(\Psi,\Box)$ and $(\Psi',\Box')$ give rise to isomorphic pairs $(X,\mathcal{L})$ and $(X',\mathcal{L}')$ if and only if there exist isomorphisms $F:M' \rightarrow M$, $\varphi:Y \rightarrow Y'$ and a linear map $A$ from $M'$ to the principal divisors on  $Y'$ such that
\[
\Box=F(\Box') \quad \text{ and } \quad \Psi'=\varphi^*F^*\Psi + A.
\]
\end{remark}

\begin{remark}
  Let $\Delta \subset M'_\QQ$ be a polytope in some lattice $M'$. Consider an exact sequence 
\[
 0 \rightarrow \ZZ \stackrel{F}{\rightarrow} M' \stackrel{G}{\rightarrow} M \rightarrow 0
\]
corresponding to the torus inclusion $T_M \hookrightarrow T_{M'}$ of codimension $1$. We choose a section $s:M \hookrightarrow M'$ and consider the map $\Psi_\Delta: G(\Delta) \rightarrow \WDiv(\PP^1)$ given by
\begin{align*}
  (\Psi_\Delta)_0(u) &= \max \{a \in \QQ \mid F_\QQ(a) + s(u) \in \Delta \cap G_\QQ^{-1}(u)\},\\
  (\Psi_\Delta)_\infty(u) &= -\min \{a \in \QQ \mid F_\QQ(a) + s(u) \in \Delta \cap G_\QQ^{-1}(u)\}.
\end{align*}
Then $(\Psi_\Delta, G(\Delta))$ is a divisorial polytope. Moreover, the construction above yields for $\Psi_\Delta$ exactly the toric variety and the ample divisor corresponding to $\Delta$ but with the restricted torus action of $T_{M}$.
\end{remark}

\begin{figure}[htbp]
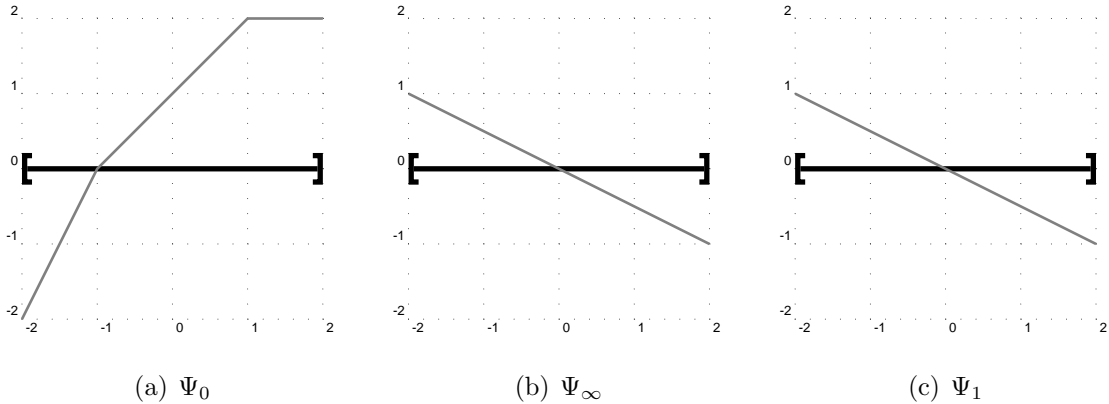

    \centering
    \subfigure[$\Psi_0$]{\hstarnull}
    \subfigure[$\Psi_\infty$]{\hstarother}
    \subfigure[$\Psi_1$]{\hstarother}
    \caption{A divisorial polytope on $\mathbb{P}^1$}\label{fig:divpoly}
\end{figure}

\begin{ex}
	Consider the divisorial polytope $\Psi$ on the interval $[-2,2]$ pictured in figure \ref{fig:divpoly}. One easily checks that the corresponding marked fansy divisor is exactly $\Xi'$ from the example in section \ref{sec:tvar}, and that the corresponding support function is exactly the function $h$ from the example in section \ref{sec:div}. Conversely, one easily checks that $h^*=\Psi$.
\end{ex}

We now describe how to read off simple geometric information of a projective $T$-variety from the corresponding divisorial polytope. For the following, we fix some divisorial polytope $(\Psi,\Box)$ with corresponding projective variety $X$ and ample divisor $D$.  We first use our divisorial polytope to define some other polytopes. 
\begin{defn}
For a finite set of points $I\subset Y$, define\begin{equation*}
\Delta(\Psi,I):=	\conv \left(\Big\{ \big(u,\sum_{P\in I}\Psi_P(u)\big)\big| u \in \Box\Big\}\bigcup\Big\{ \big(u,\sum_{P\notin I}-\Psi_P(u)\big)\big| u \in \Box\Big\}\right)\subset M_\QQ \times \QQ.
\end{equation*}
For any point $P\in Y$, define
\begin{equation*}
	\widetilde \Delta(\Psi,P):=	\conv \left(\Big\{ \big(u,\Psi_P(u)\big)\big| u \in \Box\Big\}\bigcup \Box\times \min_{u\in\Box} \Psi_P(u)\right)\subset M_\QQ \times \QQ.
\end{equation*}
Note that although $\Delta(\Psi,I)$ need not have lattice vertices, $\widetilde \Delta(\Psi,P)$ is always a lattice polytope.
\end{defn}

\begin{prop}
Let $m=\dim M_\QQ$. Then we have
$$D^{m+1}=(m+1)!\cdot \vol\Delta(\Psi,I)$$ for any set of points $I\subset Y$.
\end{prop}

\begin{proof}
	See proposition 3.31 of \cite{lars08}.
\end{proof}

For any polytope $\Delta$ with lattice vertex $v$, we say that $\Delta$ is \emph{smooth at $v$} if the directions of $\Delta$ at $v$ form a lattice basis. Now for any $P\in Y$, consider some $v\in \Box$ with $(v,\Psi_P(v))$ a vertex of the graph of $\Psi_P$.

\begin{defn}
We say that $\Psi$ is smooth at $(P,v)$ if
\begin{enumerate}
	\item For $\deg \Psi(v)>0$, $\Delta(\Psi,P)$ is smooth at $(v,\Psi_P(v))$;\label{item:nocontract}
	\item For $\deg \Psi(v)=0$, $Y=\mathbb{P}^1$ and there exist points $P_1,P_2\in Y$ such that for all points $P\neq P_1,P_2$, $(v,\Psi_P(v))$ is contained in only one full-dimensional polytope in $\Gamma_{\Psi_P}$ which additionally has integral slope, and the polytope $\Delta(\Psi,P_1)$ is smooth at $(v,\Psi_{P_1}(v))$.\label{item:contract}
\end{enumerate}
\end{defn}

\begin{prop}
	The $T$-variety $X$ corresponding to $(\Psi,\Box)$ is smooth if and only if for every $P\in Y$ and every $v\in \Box$ with $(v,\Psi_P(v))$ a vertex of $\Gamma_{\Psi_P}$, $\Psi$ is smooth at $(P,v)$.
\end{prop}

\begin{proof}
	The vertices $(v,\Psi_P(v)$ of the graphs of $\Psi_P$ correspond to affine invariant charts of the corresponding variety. If $\deg \Psi(v)>0$, then the corresponding chart has affine locus, and one easily checks that criterion \ref{item:nocontract} corresponds to the hypothesis of theorem 3.3 of \cite{suess09}. On the other hand, if $\deg \Psi(v)>0$, the corresponding chart has complete locus and the criterion \ref{item:contract} corresponds to the hypothesis of proposition 3.1 of \cite{suess09}.
\end{proof}

Finally, suppose that the divisor $D$ is in fact very ample and gives a projective embedding. We are interested in the Hilbert polynomial $\HP_D$ of $D$. Recall that for natural numbers $k$ sufficiently large, $\HP_D(k)=\dim H^0(X,k\cdot D)$. On the other hand, recall that for any lattice polytope $\Delta$ of dimension $d$, there is a unique polynomial $E_\Delta$ of degree $d$ called the Ehrhart polynomial of $\Delta$, such that for any $k \in \NN$, $E_\Delta(k)$ is the number of lattice points in $k\cdot \Delta$.

\begin{defn}Let $\mathcal{P}$ be the set of all $P\in Y$ such that $\Psi_P$ isn't trivial. We then define the \emph{Ehrhart polynomial} $E_\Psi$ of the divisorial polytope $\Psi$ by
$$
E_\Psi(k)=E_\Box(k)+\sum_{P\in \mathcal P} \left(E_{\widetilde\Delta(\Psi,P)}(k)-E_\Box(k)\cdot(1-k \min_{u\in \Box} \Psi_P(u))\right).
$$
\end{defn}
\begin{remark}
	One easily checks that if $\Psi$ only has nontrivial coefficients for two points $P_1$ and $P_2$, then $E_\Psi=E_{\Delta(\Psi,P_1)}$.
\end{remark}

\begin{prop}
	We have $$E_\Psi\geq \HP_D\geq E_\Psi-g(Y)\cdot E_\Box.$$ Furthermore, if $\deg \lfloor \Psi(u) \rfloor\geq 2g(Y)-1$ for all $u\in \Box\cap M$ then $\HP_D=E_\Psi-g(Y)\cdot E_\Box$.
	In particular, if $Y=\mathbb{P}^1$, then $\HP_D=E_\Psi$.
\end{prop}
\begin{proof}
	We have that for any $k\in\NN$, any $P\in \mathcal{P}$, and any $u\in k\cdot \Box\cap M$, $\lfloor (k\cdot\Psi)_P(u) \rfloor-k\cdot\min_{v\in \Box} \Psi_P(v)+1$ is equal to the number of lattice points in $k\cdot \widetilde \Delta (\Psi,P)$ projecting to $u$. Summing over all $u\in  k\cdot \Box\cap M$ and $P\in\mathcal{P}$, we get
	$$\sum_{u\in k\cdot \Box\cap M} \deg \lfloor (k\cdot\Psi)(u) \rfloor
	=\sum_{P\in \mathcal{P}} \left(E_{\widetilde\Delta(\Psi,P)}(k)-E_\Box(k)\cdot(1-k\cdot \min_{v\in\Box}\Psi_P(v)\right)
$$
and thus
$$\sum_{u\in k\cdot \Box\cap M} 1+\deg \lfloor (k\cdot\Psi)(u) \rfloor
	=E_\Psi(k).
$$

Now, for $k$ large enough,
$$
\HP_D(k)=\sum_{u\in k\cdot\Box\cap M} h^0(Y,(k\cdot\Psi)(u)).
$$
Applying the theorem of Riemann-Roch for curves, we have that 
$$
\deg \lfloor (k\cdot\Psi)(u) \rfloor +1-g(Y)\leq
h^0(Y,(k\cdot\Psi)(u))\leq \deg \lfloor (k\cdot\Psi)(u) \rfloor +1
$$
and the proposition follows.
\end{proof}

\begin{ex}
	We apply the above propositions three to the divisorial polytope $\Psi$ from figure \ref{fig:divpoly}. Regardless of the set of points $I\subset Y$, we always have $\vol \Delta(\Psi,I)=3$ and thus that  the corresponding divisor $D$ has self-intersection number $6$. We can also see that the corresponding projective surface is not smooth: $\Psi$ is not smooth at $(P,\pm 2)$ for any point $P\in Y$. Finally, we will see in section \ref{sec:gen} that $D$ is very ample, so we can calculate the Hilbert polynomial of $D$. Indeed, we have 
	\begin{align*}
		E_{\Box}(k)&=4k+1\\
		E_{\widetilde\Delta(\Psi,0)}&=11k^2+6k+1\\
		E_{\widetilde\Delta(\Psi,\infty)}&=
		E_{\widetilde\Delta(\Psi,1)}=4k^2+4k+1\\
	\end{align*}
and thus 
$$\HP_D(k)=E_\Psi(k)=3k^2+2k+1.$$
\end{ex}

\section{Affine Cones}\label{sec:cones}
Let $\Xi$ be a marked fansy divisor on a curve $Y$, and $h\in \CaSF(\Xi)$ such that $D_h$ is globally generated. Then the sections of $D_h$ determine a map $f:X(\Xi)\to  \mathbb{P}^n$; we denote the image of $f$ by $X$. Note that $X$ also comes with a natural complexity-one $T$-action, but in general $X$ need not be normal. By $C(X)$ we denote the affine cone over $X$ with respect to this embedding; let $\widetilde{C(X)}$ be the normalization of $C(X)$. The following proposition tells us how to describe $\widetilde{C(X)}$ in terms of a polyhedral divisor: 

\begin{prop}\label{prop:projcone}
With $h$ as above, set $$\D=\sum_P \conv (\Gamma_{-h_P})\cdot P$$
where $\Gamma_{-h_P}$ is the graph of $-h_P$. Then $\widetilde{C(X)}=X(\D)$. Furthermore, if $D_h$ is ample, then $\D$ is a proper polyhedral divisor on $Y$,
\end{prop}
\begin{proof}
	The homogeneous coordinate ring of $X$ with respect to the given embedding is $A= \bigoplus_{k\geq0} S^k(H^0(X,D_h))$, where $S^k$ is the $k$-th symmetric product. Thus, $C(X)=\spec A$. Now, the integral closure of $A$ is $\widetilde{A}=\bigoplus_{k\geq 0} H^0(X,k\cdot D_h)$, see \cite{MR0463157}, exercise II.5.14(a), and thus $\widetilde{C(X)}=\spec \widetilde{A}$.

	On the other hand, we claim that $$\bigoplus_{k\geq 0} H^0(X,k\cdot D_h)=\bigoplus_{(u,k)\in \tail(\D)\cap(M\times \ZZ)} H^0(Y,\D( (u,k))).$$
	Indeed, $H^0(X,k\cdot D_h)=\bigoplus_{u\in\Box_{k\cdot h}} H^0(Y,(k\cdot h)^*(u))$ by proposition \ref{prop:gs}. Furthermore,
	$\Box_{k\cdot h}=\{u\in M_\QQ\ |\ (u,k)\in \tail \D\}$. The claim then follows from the fact that
for $P\in Y$, 
	\begin{align*}
	\D( (u,k))_P=\min \langle \conv(\Gamma_{-h_P}),(u,k)\rangle=\min_{v\in N_\QQ} \langle v,u\rangle- k\cdot h_P(v)=(k\cdot h)^*(u).
\end{align*}
We thus have 
$$
\widetilde{C(X)}=\spec \bigoplus_{(u,k)\in \tail(\D)\cap(M\times \ZZ)} H^0(Y,\D( (u,k))) = X(\D).
$$
Now, if $D_h$ is ample, $\deg h^*(u)>0$ for $u$ in the interior of $\Box_h$, and thus $\D$ is proper.
\end{proof}

\begin{remark}
	Let $X$ be the image in projective space of some complexity-one $T$-variety $\widetilde X$ via a map corresponding to an invariant, globally generated,  ample divisor $D$. Suppose now that the normalized affine cone over $X$ is given by $\widetilde{C(X)}=X(\D)$, where $\D$ is a polyhedral divisor on some smooth projective curve $Y$ with corresponding lattice $N'$. Choose some isomorphism $N'\cong N\oplus \ZZ$, where the second term in the direct sum corresponds to the natural $\mathbb{C}^*$-action on the cone $\widetilde{C(X)}$. Reversing the above proposition, we can easily recover a marked fansy divisor $\Xi$ and a support function $h=\sum h_P\cdot P$ such that $\widetilde X=X(\Xi)$ and $D=D_h$. Indeed, let $h:N_\QQ\to \QQ$ be defined by 
$$
-h(v)_P=\min \pi_2 (\pi_1^{-1}(v)\cap \D_P)
$$
where $\pi_i$ is the projection of $N_\QQ\oplus \QQ$ onto the $i$th factor. Let $\Xi$ be the polyhedral subdivision of $N$ induced by the piecewise affine function $h$. We add marks to a top-dimensional cone $\sigma$ in the tailfan of $\Xi$ if $h_{|\sigma}(0)$ is principle, and we add marks to lower-dimensional cones $\tau$ if for some full-dimensional marked $\sigma$, $\tau\prec\sigma$ and $\deg \D^\sigma\cap\tau\neq \emptyset$. Then one easily checks that $\Xi$ is a marked fansy divisor, $h\in\CaSF(\Xi)$, $X=X(\Xi)$, and the embedding $X\hookrightarrow \mathbb{P}^n$ is given by the linear system $D_h$. Note that this procedure for determining $\Xi$ from $\D$ coincides with a special case of the procedure in section 5 of \cite{MR2363496}, although we now also retain information on the linear system $D_h$ of the embedding. The description of the corresponding divisorial polytope is even more simple: One easily checks that $\Box_h$ is the projection of $(M_\QQ \times \{1\}) \cap \tail \D ^\vee$ onto $M_\QQ$ and that $h^*(u)=\D((u,1))$.
\end{remark}

\begin{ex}
	It is not difficult to check that the divisor $D_h$ coming from the support function $h$ on $\Xi'$ of figure \ref{fig:supportfunction} is in fact globally generated. In fact, it follows from the proof of \cite{lars08} 3.27 that any semiample divisor $D_h$ is globally generated if $h^*(u)$ is globally generated for all $u\in \Box_h$. For $Y=\mathbb{P}^1$, this is always the case; thus, $D_h$ in our example is globally generated and defines a morphism to projective space with some $T$-invariant image $X$. By the above proposition, we then know that $\widetilde{C(X)}=X(\D)$, where $\tail(\D)$ is generated by $(-1,2),(1,2)$ and $\D_0$ has vertices $(-1,2),(0,1),(1,2)$, and $\D_\infty$ and $\D_1$ have sole vertex $(-1/2,0)$.
\end{ex}

\section{Finding Generators}\label{sec:gen}
Recall that for an affine toric variety coming from some pointed cone $\sigma$, a unique set of minimal generators of the corresponding multigraded algebra can be determined by calculating a Hilbert basis of $\sigma^\vee$. The goal of this section is to present a similar result for complexity-one $T$-varieties. We can then use this result to determine when a projective embedding is projectively normal.

Let $\D$ be a proper polyhedral divisor with tailcone $\sigma$ on a smooth projective curve $Y$. For $u\in\sigma^\vee\cap M$ we define $\A_u:=H^0(Y,\lfloor \D(u) \rfloor)$ and $$\A=\bigoplus_{u\in \sigma^\vee\cap M} \A_u.$$ Thus, our goal is to find generators of the $\mathbb{C}$-algebra $\A$.

Let $g$ be the genus of $Y$ and let $c$ be the minimum of $0$ and one less than the number of $P\in Y$ such that $\D_P$ is not a lattice polyhedron. Note then fur any $u\in\sigma^\vee\cap M$,
$$\deg \lfloor \D(u) \rfloor \geq \deg \D(u)-c.$$
We now take $\Sigma$ be the coarsest common refinement of  the set of all normal fans of $\D_P$, where $P$ is a point on $Y$. Note that $\D$ is linear on each cone of $\Sigma$. Each cone $\tau$ of $\Sigma$ defines a subalgebra $$\A_\tau:=\bigoplus_{u\in \tau\cap M} \A_u.$$ Note that the union of all such subalgebras is again $\A$.
For any cone $\tau\in\Sigma$, let $\tau'$ be a pointed cone and $u_\tau\in M\cap \tau\cap -\tau$ a weight such that $\tau=\tau'+\langle u_\tau \rangle$. Let $\hb(\tau')$ be the Hilbert basis of $\tau'$; note that the semigroup $\tau\cap M$ is generated by $\hb(\tau')\cup\{u_\tau\}$.
Furthermore, for $u\in\hb(\tau')\cup \{u_\tau\}$ we define $\alpha_u\in\NN$ to be the smallest number such that:
\begin{enumerate}
	\item $\D(\alpha_u\cdot u)$ is principal and $\lfloor \D(\alpha_u\cdot u) \rfloor)=\D(\alpha_u\cdot u)$; or
	\item $\alpha_u/2\in \NN$, $\deg \D(\alpha_u\cdot u)\geq 4g+2+2c$, and $\lfloor \D( (\alpha_u/2\cdot u)) \rfloor)=\D( (\alpha_u/2)\cdot u)$.
\end{enumerate}
Note that the properness of $\D$ guarantees that such an $\alpha_u$ exists. Also, some multiple of $\D(u_\tau)$ must be principle, since $\deg \D(u_\tau)=0$.

Finally, we set 
$$\G_\tau:=\Big\{ \sum_{u\in\hb(\tau')} k_u\cdot u\ \Big|\ 0\leq k_u\leq \alpha_u\Big\}\cup \Big\{ \alpha_{u_\tau}\cdot u_\tau \Big\}.$$

\begin{thm}\label{thm:genwt}
	For $\tau\in \Sigma$, the algebra $\A_\tau$ is generated in degrees $\G_\tau$. In particular, $\A$ is generated in degrees $\G_\D:=\bigcup_{\tau\in\Sigma} \G_\tau$.
\end{thm}

We will need the following lemma:
\begin{lemma}\label{lemma:surjsec}
Let $D_1$, $D_2$ be divisors on a smooth curve $Y_0$. Then the natural map 
$$
H^0(Y,D_1)\times H^0(Y,D_2)\to H^0(Y,D_1+D_2)
$$
is surjective if
\begin{enumerate}
	\item $D_1$  is principal; or
	\item $Y_0$ is complete and $\deg D_1\geq 2g+1$,$\deg D_2\geq 2g$.
\end{enumerate}
\end{lemma}
\begin{proof}
	The first case is immediate. The second case is due to Mumford \cite{MR0282975}.\end{proof}

\begin{proof}[Proof of theorem \ref{thm:genwt}]
	Fix some $\tau\in\Sigma$ and consider $u\in \tau'\cap M$ such that $u\notin\G_\tau$. Then there is some $u'\in \hb(\tau')$ and $u''\in \tau'\cap M$ such that $u=\alpha_{u'} u'+u''$.
Suppose first that $\D(\alpha_{u'} u')$ is principal. Then 
$$\lfloor \D(u)\rfloor= \lfloor \D(\alpha_{u'}u') \rfloor + \lfloor \D(u'') \rfloor$$ and it follows from the first case of the above lemma that $\A_{u}$ is generated by $\A_{\alpha_{u'}u'}$ and $\A_{u''}$.
Suppose now instead that $\deg \alpha_{u'} \D(u')\geq 4g+2+c$. Then 
$$\lfloor \D(u)\rfloor= \lfloor \D( (\alpha_{u'}/2)u') \rfloor + \lfloor \D((\alpha_{u'}/2)u'+u'') \rfloor$$ and we have 
\begin{align*}
	\deg \lfloor \D( (\alpha_{u'}/2)u') \rfloor&\geq 2g+1+c\\
	\deg \lfloor \D((\alpha_{u'}/2)u'+u'') \rfloor&\geq 2g+1+c-c\geq 2g+1.
\end{align*} Thus, it follows from the second case of the above lemma that $\A_{u}$ is generated by $\A_{(\alpha_{u'}/2)u'}$ and $\A_{(\alpha_{u'}/2)u'+u''}$. 
Continuing this argument by induction, we can conclude that $\A_u$ is generated in degrees lying in $\G_\tau$ for any $u\in \tau'\cap M$. 

Now consider any $u\in \tau\cap M$ such that $u\notin\G_\tau$. We can write $u=k\alpha_{u_\tau}u_\tau+u'$ for some $u'\in \tau'\cap M$ and $k\in \NN$. Then once again by the first part of the above lemma, $\A_{u}$ is generated by $\A_{k\alpha_{u_\tau}u_\tau'}$ and $\A_{u'}$, but by the above we already knew that $\A_{u'}$ is generated by degrees in $\G_\tau$. Thus, we can conclude that $\A_\tau$ is generated in degrees $\G_\tau$. The statement concerning $\A$ follows immediately.
\end{proof}

We can now use theorem \ref{thm:genwt} to give a finite list of generators of $\A$. Note that we can consider $\A_0$ as a finitely generated $\mathbb{C}$-algebra, say with generators $f_0^1\ldots,f_0^{d_0}$. Now for $u\in\G_\D$, $u\neq 0$, let $f_u^1,\ldots,f_u^{d_u}$ generate $\A_u$ as an $\A_0$-module. The following corollary is then immediate:

\begin{cor}
	The algebra $\A$ is generated as a $\C$-algebra by $$\Big\{f_u^i\Big\}_{\substack {u\in\G_\D \\ 1\leq i \leq d_u}}.$$
\end{cor}

The weight set $\G_\D$ and set of generators of $\A$ from the above corollary are in general not minimal. However, it is immediately clear that if the tail cone of $\D$ is full-dimensional, 
$$
\G_\D^{\min}:=\G_\D\setminus\left\{ u \in\G_\D \ \Big |\ \sum_{\substack{ u'\in\G^\D\\ u-u'\in\sigma^\vee}} H^0(Y,\lfloor\D(u')\rfloor)\times H^0(Y,\lfloor\D(u-u')\rfloor)=H^0(Y,\lfloor\D(u)\rfloor)\right\}
$$
is the unique minimal set of weights needed to generate $\A$. This set can be constructed by checking a \emph{finite} number of conditions.

\begin{cor}
	Let $\Xi$ be a marked fansy divisor on a curve $Y$, and $h\in\CaSF(\Xi)$ such that $D_h$ is very ample. Then the corresponding embedding is projectively normal if and only if all elements of $\G_\D^{\min}$ have last coordinate equal to one, where  $\D$ is defined as in proposition \ref{prop:projcone}. 
\end{cor}
\begin{proof}
The embedding is projectively normal if and only if $\A$ is generated in degree one with respect to the relevant $\ZZ$-grading.
\end{proof}

\begin{ex}
	Consider the polyhedral divisor $\D$ from the example of the previous section. One easily checks that $\G_\D=\{(-2,1),(-1,1),(0,1),(1,1),(2,1)\}$. Thus, the image $X$ of $X(\Xi)$ under the linear system $|D_h|$ is projectively normal. It then follows that $D_h$ is in fact very ample and the corresponding map is actually an embedding. Indeed, on the one hand, the quotient of $C(X)$ by $\mathbb{C}^*$ is clearly $X$. On the other hand, one can also check by calculation that the quotient of $C(X)$ by $\mathbb{C}^*$ is $X(\Xi)$.
\end{ex}

\begin{remark}
	If instead of being a curve, we have $Y=\mathbb{P}^n$ or $Y=\mathbb{A}^n$, we can define a similar set $\G_\D$ similar to above containing the weights generating $\A$. Indeed, in both cases there are statements similar to lemma \ref{lemma:surjsec}.
\end{remark}

\bibliography{proj-tvar}
\address{
Nathan Owen Ilten\\
Mathematisches Institut\\
Freie Universit\"at Berlin\\
Arnimallee 3\\
14195 Berlin, Germany}{nilten@cs.uchicago.edu}

\address{Hendrik S\"u\ss{}\\
	Institut f\"ur Mathematik\\
        LS Algebra und Geometrie\\
        Brandenburgische Technische Universit\"at Cottbus\\
        PF 10 13 44\\
        03013 Cottbus, Germany}{suess@math.tu-cottbus.de}

\end{document}